\documentclass[oneside,a4paper,11pt,notitlepage]{article}
\usepackage[left=0.8in, right=0.8in, top=1.5in, bottom=1.5in]{geometry}
\usepackage{comment}
\usepackage[T1]{fontenc} 
\usepackage[utf8]{inputenc} 
\usepackage[english]{babel}
\usepackage{amssymb}
\usepackage{lipsum} 
\usepackage{lmodern}
\usepackage{amsmath}
\usepackage{amsthm}
\usepackage{bm}
\usepackage{mathtools}
\usepackage{braket}
\usepackage{esint}
\newcommand{\abs}[1]{{\left|#1\right|}}
\newcommand{\norma}[1]{{\left\Vert#1\right\Vert}}

\usepackage{enumitem}
\usepackage{booktabs}
\usepackage{graphicx}
\usepackage{tikz}
\usetikzlibrary{patterns}
\usepackage{multicol}
\usepackage{caption}
\usepackage{bigints}
\usepackage[skins,theorems]{tcolorbox}
\tcbset{highlight math style={enhanced,
colframe=black,colback=white,arc=0pt,boxrule=1pt}}
\def\XXint#1#2#3{{\setbox0=\hbox{$#1{#2#3}{\int}$}
    \vcenter{\hbox{$#2#3$}}\kern-.5\wd0}}

\theoremstyle{definition}
\newtheorem{definizione}{Definition}[section]
\theoremstyle{plain}
\newtheorem{teorema}{Theorem}[section]

\newtheorem{lemma}[teorema]{Lemma}
\newtheorem{prop}[teorema]{Proposition}
\newtheorem{corollario}[teorema]{Corollary}
\theoremstyle{definition}
\newtheorem{esempio}{Example}[section]
\newtheorem{oss}[esempio]{Remark}

\newtheorem*{open*}{Open problems}
\DeclareMathOperator{\R}{\mathbb{R}}

\makeatletter
\newcommand{\myfootnote}[2]{\begingroup
	\def\@makefnmark{}%
	\addtocounter{footnote}{-1}%
	\footnote{\textbf{#1} #2}
	\endgroup}
\makeatother

\usepackage{hyperref}
\hypersetup{linktoc=none, bookmarksnumbered, colorlinks=true, linkcolor=magenta, citecolor=cyan}

\title{
A quantitative result for the $k$-Hessian equation}
\author{Alba Lia Masiello*, Francesco Salerno}
\date{}

\newcommand{\Addresses}{{
\bigskip 
  
   \medskip

    \textit{E-mail address}, A.L.~Masiello* (corrisponding author): \texttt{albalia.masiello@unina.it} 
  
   \medskip 
 \textsc{Dipartimento di Matematica e Applicazioni ``R. Caccioppoli'', Universit\`a degli studi di Napoli Federico II, Via Cintia, Complesso Universitario Monte S. Angelo, 80126 Napoli, Italy.}
   
     \textit{E-mail address}, F.~Salerno: \texttt{f.salerno@ssmeridionale.it} 
   \medskip

 \textsc{Mathematical and Physical Sciences for Advanced Materials and Technologies, Scuola Superiore Meridionale, Largo San Marcellino 10, 80126 Napoli, Italy.}

 \par\nopagebreak 

}} 

\begin{document}
\maketitle

\begin{abstract}

{In this paper, we study a symmetrization that preserves the mixed volume of the sublevel sets of a convex function, under which, a  P\'olya-Szeg\H o type inequality holds. We refine this symmetrization to obtain a quantitative improvement of the P\'olya-Szeg\H o inequality for the $k$-Hessian integral,   and, with similar arguments, we show a quantitative inequality for the comparison proved by Tso \cite{tso} for solutions to the $k$-Hessian equation.} 

 As an application of the first result, we prove a quantitative version of the Faber-Krahn and Saint-Venant inequalities for these equations. 
\end{abstract}

\section{Introduction}

{In recent decades, symmetrization techniques have been extensively studied in relation to the qualitative properties of solutions to second-order elliptic boundary value problems. A classic example is the Schwarz symmetrization, used in spectral analysis, among others, to provide a solution 
 to very classical problems such as the Faber-Krahn inequality or the Saint-Venant inequality. For instance, Talenti uses the Schwarz symmetrization in \cite{talenti76} to prove sharp a priori bounds for solutions to the Poisson equation with Dirichlet boundary conditions. The peculiarity of Schwarz symmetrization is that it preserves the volumes of the superlevel sets of a function, producing a radially symmetric function defined on a ball that has the same 
$L^p$ norms as the original.   
In  \cite{talenti_monge_ampere}, Talenti introduces a new type of symmetrization that preserves the perimeter of the level sets of a function. This symmetrization suitably adapts to the Monge-Ampère problem in the plane, that is

\begin{equation}\label{talma}
    \begin{cases}
            \mathrm{det}(D^2 u)=f & \text{ in } \Omega,\\
        u=0 & \text{ on } \partial\Omega,
    \end{cases}
\end{equation}
where $\Omega\subset\mathbb{R}^2$ is a bounded, open and convex set and $f$ is a positive function. The author proves that it is possible to compare the solution to  \eqref{talma} with the solution to 

\begin{equation}\label{talmasym}
    \begin{cases}
        \mathrm{det}(D^2 v)=f^0 & \text{ in } \Omega^*_1,\\
        v=0 & \text{ on } \partial\Omega^*_1,
    \end{cases}
\end{equation}
where $f^0$ is the Schwarz symmetrization of $f$ and $\Omega^*_1$ is the ball with the same perimeter as $\Omega$. 

Some years later, Tso in \cite{tso} generalizes the result by Talenti in any dimension, and the setting is the following.

Let $\Omega$ be an open, bounded and convex set, let $u\in C^2(\Omega)\cap C_0(\overline{\Omega})$ be a convex and negative function, and let us consider the following functional

\begin{equation}\label{dirichlet}
    H_k(u;\Omega)=\frac{1}{k+1}\int_\Omega(-u)S_k(D^2u) \, dx,
\end{equation}
where $S_k(D^2 u)$ is the $k$-Hessian operator, defined as

\begin{equation}\label{kess}
    S_k(D^2 u)= \sum_{1\le i_1< \dots < i_k\le n} \lambda_{i_1}\cdots\lambda_{i_k}, \quad k=1,\dots, n,
\end{equation}
being $\lambda_i$ the eigenvalues of the Hessian matrix of $u$. We notice that $S_k$ is a second-order differential operator and it reduces to the Monge-Ampère operator for $k=n$ and to the Laplace operator for $k=1$, when the functional $H_1$ coincides with the Dirichlet energy. 
Tso proves that, for all $k$, there exists a symmetrization that decreases the functional \eqref{dirichlet}. Moreover, the author exhibits this symmetrization that preserves the $(k-1)$-th quermassintegral of the sublevelsets of the function $u$. For a convex set $\Omega$, the quermassintegrals are geometric quantities that characterize the shape of the set itself, the precise definition can be found in  \S\ref{quermass}, while the definition of the associated symmetrization, which we will refer to as $(k-1)$-symmetrized, in \S\ref{symmetrization}. 

In the same paper, the author also generalizes the comparison by Talenti for the solution to the Monge-Amp\'ere equation to the $k$-Hessian equation in any dimension, proving that, {whenever the problem }
\begin{equation}
    \label{poisson:sk}
    \begin{cases}
        S_k(D^2 u)=f & \text{in } \Omega,\\
        u=0 & \text{ on } \partial\Omega,
    \end{cases}
\end{equation}
where $f$ is a positive function, admits a convex solution, then
the $(k-1)$-symmetrized of the solution $u$
 can be pointwise compared with the solution to 
\begin{equation}
        \label{poisson:symm}
        \begin{cases}
            S_k(D^2u^0)=f^0& \text{in } \Omega_{k-1}^*,\\
            u^0=0& \text{on } \partial\Omega_{k-1}^*,
        \end{cases}
    \end{equation}
    where $f^0$ is the Schwarz symmetrization of $f$ and $\Omega^*_{k-1}$ is the ball with the same $(k-1)$-th quermassintegral, 
obtaining

\begin{equation}
    \label{poitso}
    u^0(x)\le u^*_{k-1}(x)\le 0, \quad \forall x \in \Omega^*_{k-1}.
\end{equation}

{It is not true, in general, that the solution to \eqref{poisson:sk} is a convex function. Indeed, it was proved in \cite{Wang} that if  $f$ is a positive $C^2$ function, then the problem \eqref{poisson:sk} admits a $k$-convex solution $u\in C^2(\Omega)$, in the sense that}

$$S_i(D^2u)\ge 0, \quad \forall i=1,\cdots, k.$$

Neverthless, as already highlighted in \cite{Nunziask, salani_logcon}, the result \eqref{poitso} can be obtained by relaxing the assumption on $u$ and only requiring that it has convex sublevel sets. Some evidence in proving the convexity of the solution are obtained in \cite{Ma2008TheCO, salani_logcon}.

The result by Tso is important not only because it introduces a new symmetrization, but also because it allows one to treat some optimization problems. Indeed, as a consequence of the results by Tso, it is possible to prove a Faber-Krahn type inequality for the eigenvalues of the $k$-Hessian operator and a Sain-Venant type inequality for the $k$-th torsional rigidity in a special class of convex sets. These topics interested many authors over the years, we refer for example to \cite{ BNT_newisoperimetric, Brandolini2007ComparisonRF, DPG_stab_eig, DPG_skautovalori, Gavitone2010WeightedEP, GS_stability}. For this reason, these results have also been extended in the anisotropic setting in \cite{DPG_anisotr_per, DPGX_anisotropo}.

 \vspace{2mm}
 The present work aims to improve the results by Tso in a quantitative way, following in the footsteps of the recent paper \cite{ABMP}, in which the authors provide a quantitative version of the Talenti comparison in \cite{talenti76}. 
 The main tool in proving the results by Tso is an isoperimetric inequality for the quermassintegrals, known as Alexandrov-Fenchel inequality (see \S\ref{quermass}). 
 So, to obtain a quantitative version the key role is played by the quantitative quermassintegral isoperimetric inequality due to Groemer and Schneider \cite{Groemer_schneider91}.} In this work, the authors bound from below the isoperimetric deficit in terms of the Hausdorff asymmetry (see Section \ref{sec2}, Definition \ref{hauss}, for the precise definition) $\alpha_\mathcal{H}(\Omega)$, an index that measures the $L^\infty$ distance of $\Omega$ to a ball. 
 More precisely, our main results are
\begin{teorema}
\label{theorem1.1}
    Let $\Omega$ be an open, bounded, and convex set of $\mathbb{R}^n$ and let $u\in C^2(\Omega)$ be a convex function that vanishes on the boundary of $\Omega$. Then, there exists a positive constant $C_1=C_1(n,k,\Omega)$ such that
     \begin{equation}\label{polyask}
       \dfrac{ H_k(u;\Omega)-H_k(u_{k-1}^*;\Omega_{k-1}^*)}{\norma{u}_{L^\infty(\Omega)}^{k+1}}\geq C_1\alpha_\mathcal{H}^{\frac{n+3}{2}+k+1}(\Omega),
    \end{equation}

    where $k=1,\dots,n-1$,
    
    $$C_1(n,k,\Omega)=c_1(n,k)\zeta_{k-1}^{n-2k}(\Omega) ,$$

    $$c_1(n,k)=\binom{n-1}{k-1}\left(\frac{\omega_{n-1}}{n\omega_n}\right)^{k+1}\frac{(n+1)\beta_n}{2^\frac{n+3}{2}k(n-1)(n-k+1)}\left(\frac{2n+3}{2(n+2)}     \right)^{n-k},$$

     $\beta_n$ is given in Theorem  \ref{quantitative:af}, and $\zeta_{k-1}(\Omega)$ is the $(k-1)$-th meanradius of $\Omega$.
\end{teorema}
Theorem \ref{theorem1.1} allows us to improve the Faber-Krahn and Saint-Venant-type inequalities in a quantitative form as stated in Corollaries \ref{quant:eigen}-\ref{quant:tors}. 

Later, we focus on the Poisson equation \eqref{poisson:sk} and we prove a quantitative version of inequality \eqref{poitso} with the same techniques of Theorem \ref{theorem1.1}. 

\begin{teorema}\label{teo1.2}
    Let $\Omega$ be an open, bounded and convex set of $\mathbb{R}^n$ and let $f$ be a positive and measurable function. Let $u\in C^2(\Omega)$ be a convex solution to (\ref{poisson:sk}), $u_{k-1}^*$ be its $(k-1)$-symmetrized, $k=1,\dots,n-1$, and let $u^0$ be the solution to (\ref{poisson:symm}). Then there exists a positive constant $C_2=C_2(n,k)$ such that
    \begin{equation*}
        \frac{\|u_{k-1}^*-u^0\|_{L^\infty(\Omega_{k-1}^*)}}{\|u\|_{L^\infty(\Omega)}}\geq C_2\alpha_\mathcal{H}^\frac{n+5}{2}(\Omega),
    \end{equation*}
    where
    \begin{equation*}
        C_2(n,k)=\frac{(n+1)(k+1)\beta_n\omega_{n-1}}{2^\frac{n+9}{2}n^2(n+2)(n-1)(n-k+1)\omega_n^2}.
    \end{equation*}
\end{teorema}

 Theorem \ref{teo1.2} recalls the main result in \cite{ABMP}, and also its extension to the Hamilton-Jacobi equation contained in \cite{amato2024quantitativecomparisonresultsfirstorder}; here the advantage is that we only consider convex functions defined on convex sets.

The pointwise comparison \eqref{poitso} between the symmetrized of the solution to \eqref{poisson:sk} and the solution to \eqref{poisson:symm}
can be used to prove a comparison between the $k$-Hessian functional of $u$ and $u^0$. We prove the following
\begin{teorema}
\label{gradienti}
     Let $\Omega$ be an open, bounded and convex set of $\mathbb{R}^n$ and let $f$ be a positive and measurable function. Let $u\in C^2(\Omega)$ be a solution to \eqref{poisson:sk} with convex sublevel sets and let $u^0$ be the solution to \eqref{poisson:symm}, then it holds

    \begin{equation}
        \label{dis:hk}
        H_k(u;\Omega)\le H_k(u^0; \Omega^*_{k-1}),    \end{equation}

    for all $k=1,\dots,n$.
\end{teorema}

In Section \ref{sec3}, we make use of Theorem \ref{teo1.2} to improve inequality \eqref{dis:hk} in a quantitative way.

\begin{teorema}
    \label{teo1.3}
    Let $\Omega$ be an open, bounded and convex set of $\mathbb{R}^n$ and let $f$ be a positive and measurable function. Let $u$ be a convex solution to \eqref{poisson:sk} and let $u^0$ be the solution to \eqref{poisson:symm}. Then there exists a positive constant $C_3=C_3(n,k)$ such that

    \begin{equation*}       
        \frac{H_k(u^0; \Omega^*_{k-1})-H_k(u;\Omega)}{\norma{u}_{L^\infty(\Omega) }\norma{f}_{L^1(\Omega)}}\ge C_3\alpha_\mathcal{H}^\frac{n+5}{2}(\Omega),
    \end{equation*} 
    where
    \begin{equation*}
        C_3(n,k)=\left(\frac{2n+3}{2(n+2)}\right)^n \frac{\omega_{n-1}\beta_n}{n^2\omega_n^2} \frac{(n+1)(k+1)}{k(n-1)(n-k+1)}  2^{-\frac{n+7}{2}}
    \end{equation*}
    
    where $\beta_n$ is given in Theorem \ref{quantitative:af} and $k=1,\dots,n-1$.
\end{teorema}

    We note that the case $k=n$, namely the Monge-Ampère equation, is not covered in this work, except for Theorem \ref{gradienti}. In particular, in the first result, we cannot apply the Aleksandrov-Fenchel inequality for such index. Additionally, we cannot prove Theorem \ref{teo1.2} using the same technique as in the cases $k=1,\dots,n-1$, even if some partial results can be obtained. So, it is clear that Theorem \ref{teo1.3} cannot be extended as a consequence of Theorem \ref{teo1.2}.

The paper is organized as follows: in Section \ref{sec2} we recall some preliminary tools about convex geometry and quermassintegral symmetrization; in Section \ref{sec3} we prove the main Theorems.

\section{Notation and preliminaries}\label{sec2}
We provide the classical definitions and results that we need in the following. The reader can refer to \cite[Chapter 1]{Schneider_2013} for more details.
\begin{definizione}
    \label{minksum}
    Let $\Omega, K\subset \R^2$ two bounded convex sets. We define the \emph{Minkowski sum} $(+)$ as
    \begin{equation*}
        \label{sum}
        \Omega+K:=\{x+y \, : \, x\in \Omega, \, y\in K\},
    \end{equation*}
\end{definizione} 

\begin{definizione}
    Let $E,F$ be two convex sets in $\mathbb{R}^n$. The \emph{Hausdorff distance} between $E$ and $F$ is defined as
\[
    d_{\mathcal{H}}(E,F) := \inf\left\{ \varepsilon>0 \, : \,  E \subset F + \varepsilon B , \, F \subset E + \varepsilon B  \right\},
\]
where $B$ is the unitary ball centered at the origin and $F + \varepsilon B$ is the Minkowski sum.
\end{definizione}

\begin{definizione}
    \label{support}
Let $\Omega$ be a convex set. The \emph{support} function of $\Omega$ is defined as
\begin{equation*}
        h(\Omega, u):=\max_{x\in \Omega} (x\cdot u), \qquad u\in \mathbb{S}^{n-1}.
    \end{equation*}
\end{definizione}
The support function allows us to define, for a convex set $\Omega$, two geometrical quantities.
\begin{definizione}
 The \emph{width function} $w(\Omega, \cdot)$ of $\Omega$
is defined as
$$w(\Omega, u):= h(\Omega, u) + h(\Omega, -u) \quad \text{for } u\in \mathbb{S}^{n-1}.$$

The quantity $w(\Omega, u)$ is the thickness of $\Omega$ in the direction $u$ and it represents the distance between
the two support hyperplanes of $\Omega$ orthogonal to $u$. The maximum of the width function,
$$D(\Omega):=\max_{u\in \mathbb{S}^{n-1}} w(\Omega, u)$$
is the diameter of $\Omega$.

The mean value of the width function is called the \emph{mean width} and it is denoted by 
$$w(\Omega):= \frac{2}{n\omega_n} \int_{\mathbb{S}^{n-1}}h(\Omega, u)\, du.$$

The \emph{Steiner point} $s(\Omega)$ is defined via the vector-valued integral

$$s(\Omega):=\frac{1}{\omega_n}\int_{\mathbb{S}^{n-1}}h(\Omega, u)u\, du$$
\end{definizione}

\begin{definizione}
    \label{steinerball}
    The \emph{Steiner ball} $B_\Omega$ of a convex set $\Omega$ is the ball centered at the Steiner point of $\Omega$ with diameter equal to $w(\Omega)$.
\end{definizione}

\begin{oss}
    Thanks to the Definition \ref{support} of the support function of a convex set, it is possible to characterize the Hausdorff distance as follows

    $$d_\mathcal{H}(E, F)= \norma{h(E,\cdot)-h(F,\cdot)}_{L^\infty}.$$
    Now it is clear that
    \begin{equation*}
        \abs{s(E)-s(F)}\le n d_\mathcal{H}(E,F).
    \end{equation*}
\end{oss} 
\begin{oss}
    \label{width/diam}
    Let us observe that it is possible to compare the mean width and the diameter of $\Omega$, for any convex set $\Omega\subset\mathbb{R}^n$, as the following bounds hold

    $$\frac{2\omega_{n-1}}{n\omega_n}\le\frac{w(\Omega)}{D(\Omega)}\le 1.$$

    The upper bound follows by the definition of diameter and mean width, while the lower bound is proved in \cite{alexandrov}, and we recall here the proof for the reader's convenience.
     
     Let $S\subset\Omega$ be the diameter segment and let us suppose, without loss of generality, that it can be  written as
    \begin{equation*}
        S=\frac{D(\Omega)}{2}\alpha \textbf{e}_1,\quad \alpha\in[-1,1]
    \end{equation*}
    Then, by the definition of support function, we have that
    \begin{equation*}
        h(\Omega,u)\geq h(S,u)=\frac{D(\Omega)}{2}\sup_{\alpha\in[-1,1]}\langle \alpha\textbf{e}_1, u \rangle=\frac{D(\Omega)}{2}|\langle \textbf{e}_1,u\rangle|
    \end{equation*}
    and then, if we write $\overline{u}=(u_2,\dots,u_n)$
    \begin{equation*}
        \begin{split}
            \frac{w(\Omega)}{D(\Omega)}&=\frac{1}{D(\Omega)}\frac{2}{n\omega_n}\int_{\mathbb{S}^{n-1}}h(\Omega,u)\,du\geq \frac{1}{n\omega_n}\int_{\mathbb{S}^{n-1}}|\langle \textbf{e}_1,u\rangle|\,du=\frac{1}{n\omega_n}\int_{\{u_1^2+\dots+u_n^2=1\}}|u_1|\,d\mathcal{H}^{n-1}(u)\\
            &=\frac{2}{n\omega_n}\int_{B^{n-1}} \sqrt{1-\abs{\overline{u}}^2}\frac{1}{\sqrt{1-\abs{\overline{u}}^2}}\, d\overline{u}= \frac{2\omega_{n-1}}{n\omega_n}.
        \end{split}
    \end{equation*}
\end{oss}
\subsection{Quermassintegrals}  \label{quermass}
 For the content of this section, we will refer to \cite{Schneider_2013}. Let $\Omega \subset \R^n$ be a non-empty, bounded, convex set, let $B$ be the unitary ball centered at the origin and $\rho > 0$. We can write the Steiner formula for the Minkowski sum $\Omega+ \rho B$ as
\begin{equation}
    \label{Steiner_formula}
    \abs{\Omega + \rho B} = \sum_{i=0}^n \binom{n}{i} W_i(\Omega) \rho^{i} .
\end{equation}
The coefficients $W_i(\Omega)$ are known in the literature as quermassintegrals of $\Omega$. In particular, $W_0(\Omega) = \abs{\Omega}$,  $nW_1(\Omega) = P(\Omega)$ and $W_n(\Omega) = \omega_n$ where $\omega_n$ is the measure of $B$.

Formula \eqref{Steiner_formula} can be generalized to every quermassintegral, obtaining
\begin{equation} \label{Steinerquermass}
    W_j(\Omega+\rho B) = \sum_{i=0}^{n-j} \binom{n-j}{i} W_{j+i}(\Omega) \rho^i, \qquad j=0, \ldots, n-1.
\end{equation}

If $\Omega$ has $C^2$ boundary, the quermassintegrals can be written in terms of principal curvatures of $\Omega$. More precisely, denoting with $\sigma_k$ the $k$-th normalized elementary symmetric function of the principal curvature $\kappa_1, \ldots, \kappa_{n-1}$ of $\partial \Omega$, i.e.
 \[
 \sigma_0 = 1, \qquad \qquad \sigma_j = \binom{n-1}{j}^{-1} \sum_{1 \leq i_1 < \ldots < i_j \leq n-1} \kappa_{i_1} \ldots \kappa_{i_j}, \qquad j = 1,\ldots,n-1,
 \]
 then, the quermassintegrals can be written as
 \begin{equation}
     \label{quermass_con_curvature}
     W_j(\Omega) = \frac{1}{n} \int_{\partial \Omega} \sigma_{j-1} \, d \mathcal{H}^{n-1}, \qquad j = 1,\ldots, n-1.
 \end{equation}

Furthermore, Aleksandrov-Fenchel inequalities hold true
\begin{equation}
    \label{Aleksandrov_Fenchel_inequalities}
    \biggl( \frac{W_j(\Omega)}{\omega_n} \biggr)^{\frac{1}{n-j}} \geq \biggl( \frac{W_i(\Omega)}{\omega_n} \biggr)^{\frac{1}{n-i}}, \qquad 0 \leq i < j \leq n-1,
\end{equation}
where equality holds if and only if $\Omega$ is a ball. When $i=0$ and $j=1$, formula \eqref{Aleksandrov_Fenchel_inequalities} reduces to the classical isoperimetric inequality, i.e.
\[
    P(\Omega) \geq n \omega_n^{\frac{1}{n}} \abs{\Omega}^{\frac{n-1}{n}}.
\]
As the classical isoperimetric inequality, the Alexandrov-Fenchel inequalities can be improved in a quantitative form, as proved in \cite{Groemer_schneider91}.

\begin{teorema}[Groemer-Schneider]
    \label{quantitative:af}
    Let $\Omega$ be a bounded convex set with Steiner ball $B_\Omega$, then, for $0\le i<j\le n-1$, it holds
    \begin{equation}
        \label{quantitative:quermass}
       \begin{aligned}
           &\dfrac{\omega_n^{i-j} W_j(\Omega)^{n-i}-W_i(\Omega)^{n-j}}{W_i(\Omega)^{n-j}}\ge \\ &\frac{n+1}{n(n-1)}s_{j-i}\left(\frac{W_{n-1}^2(\Omega)}{\omega_n}, W_{n-2}(\Omega)\right)\frac{\beta_n}{W_{n-2}^{j-i}(\Omega)} \left(\dfrac{\omega_n }{W_{n-1}(\Omega)}\right)^{\frac{n-1}{2}}d_{\mathcal{H}}(\Omega,B_\Omega)^{\frac{n+3}{2}},
       \end{aligned}       
    \end{equation}
    where 
    $$s_{m}(x,y)=\sum_{\nu=0}^{m-1}x^\nu y^{m-\nu-1}, \quad \beta_n=\alpha_n\left(\frac{n\omega_n}{\omega_{n-1}}\right),$$

    and
    $$\alpha_n(c)=\frac{\omega_{n-1}}{(n+1)(n+3)}\min \left\{\frac{3}{\pi^2 n(n+2)2^n}, \frac{16(c+2)^\frac{n-3}{2}}{(c-1)^{n-2}}\right\}.$$
\end{teorema}

\begin{oss}
    Let us observe that by the Alexandrov-Fenchel inequalities \eqref{Aleksandrov_Fenchel_inequalities} and the fact that $j-i\ge 1$, the quantitative result \eqref{quantitative:quermass} can be rewritten as follows  
    \begin{equation}
        \label{alexfenchshort}
        \dfrac{\omega_n^{i-j} W_j(\Omega)^{n-i}-W_i(\Omega)^{n-j}}{W_i(\Omega)^{n-j}}\ge \frac{(n+1)\beta_n}{n(n-1)\omega_n}\left(\frac{\omega_n d_{\mathcal{H}}(\Omega,B_\Omega)}{W_{n-1}(\Omega)} \right)^\frac{n+3}{2}.
    \end{equation}
\end{oss}

\begin{definizione}
    Let $\Omega$ be a compact convex set, we define the \textit{k-th mean radius}, $k=1,\dots,n-1$, of $\Omega$ as
\begin{equation}
    \label{meanradii}
    \zeta_k(\Omega)=\left(\frac{W_k(\Omega)}{\omega_n}\right)^\frac{1}{n-k}.
\end{equation}
\end{definizione}
The quantitative Alexandrov-Fenchel inequalities can be rewritten in terms of mean radii, obtaining
\begin{equation}
    \label{quantitative_mean_radii}
    \frac{\zeta_{j}^{(n-j)(n-i)}-\zeta_{i}^{(n-j)(n-i)}}{\zeta_{i}^{(n-j)(n-i)}} \ge \frac{(n+1)\beta_n}{n(n-1)\omega_n}\left(\frac{ d_{\mathcal{H}}(\Omega,B_\Omega)}{\zeta_{n-1}(\Omega)} \right)^\frac{n+3}{2}.
\end{equation}
\begin{oss}
    We observe that, since $w(\Omega)=2\zeta_{n-1}(\Omega)$, Remark \ref{width/diam} implies 
    
    \begin{equation}
        \label{zeta/diam}
        \frac{\zeta_{n-1}(\Omega)}{D(\Omega)}\geq \frac{\omega_{n-1}}{n\omega_n}.
    \end{equation}
\end{oss}
\begin{oss}
    If $K$ is a $j$-dimensional convex set in $\R^n$, then the $i$-th quermassintegral, $i=1,\dots,n-j-1$, equals zero, while the others may be non-null.
\end{oss}

\begin{definizione}\label{hauss}
The \textit{Hausdorff asymmetry index} is defined as 
\begin{equation*}
    \alpha_\mathcal{H}(\Omega)=\dfrac{d_\mathcal{H}(\Omega,B_\Omega)}{\zeta_{n-1}(\Omega)},
\end{equation*}
 where $B_\Omega$ is the Steiner ball of $\Omega$.  
\end{definizione}

It is possible to prove that the Hausdorff asymmetry index is bounded: indeed if we denote by
\begin{equation*}
    \overline{r}=d(s_\Omega,\partial\Omega),
\end{equation*}
it follows that $\overline{r}>0$ as the Steiner point $s(\Omega)\in \mathring{\Omega}$ (see \cite{HCG}), so $B_{\overline{r}}(s(\Omega))\subset \Omega$, and 

\begin{equation*}
    d_\mathcal{H}(\Omega,B_\Omega)\leq d_\mathcal{H}(B_{\overline{r}},B_\Omega)=\zeta_{n-1}(\Omega)-\overline{r}=\leq \zeta_{n-1}(\Omega),
\end{equation*}
that is 
\begin{equation}
    \label{3}
    \alpha_\mathcal{H}(\Omega)\leq 1.
\end{equation}

We now want to prove a "Hausorff counterpart" of the propagation of the Fraenkel asymmetry contained in \cite{brasco}, as we aim to adapt their techniques, inspired by an idea contained in \cite{hansen}, to our case. 

\begin{lemma}[Propagation of the Hausdorff asymmetry]
    \label{lembrasco}
    Let $\Omega\subset\mathbb{R}^n$ be a bounded, convex set with finite measure and let $U\subset\Omega$, $\abs{U}>0$ be such that 
    \begin{equation}\label{cond}
       d_\mathcal{H}(\Omega, U)\leq \dfrac{1}{2(n+2)}d_\mathcal{H}(\Omega, B_\Omega),
     \end{equation}
    where $B_\Omega$ is the Steiner ball of $\Omega$. 
    Then, we have
        \begin{equation}\label{claim_lem}
            d_\mathcal{H}(U, B_U)\geq \dfrac{1}{2}d_\mathcal{H}(\Omega, B_\Omega),
        \end{equation}
    where $B_U$ is the Steiner ball 
    of $U$.
\end{lemma}
\begin{proof}
    Let us set
    \begin{equation*}
        R=\zeta_{n-1}(\Omega),\quad r=\zeta_{n-1}(U),
    \end{equation*}
    and let us denote by $\varepsilon_1=d_\mathcal{H}(\Omega, U)$, $\varepsilon_2=d_\mathcal{H}(B_U,B_R)$, then,
     the monotonicity with respect to the inclusion of the mean radii and the Steiner formula for the quermassintegral \eqref{Steinerquermass} give
    \begin{equation*}
        \begin{split}
            R&=\zeta_{n-1}(\Omega)\leq \zeta_{n-1}(U+\varepsilon_1 B_1)=\frac{W_{n-1}(U+\varepsilon_1 B_1)}{\omega_n}=\frac{W_{n-1}(U)+\omega_n\varepsilon_1}{\omega_n}
            =\frac{W_{n-1}(U)}{\omega_n}+\varepsilon_1\\
            &=r+\varepsilon_1=r+d_\mathcal{H}(\Omega, U).
        \end{split}
    \end{equation*}
    Moreover, 
    \begin{equation}
        \label{distpalle}
        d_\mathcal{H}(B_R, B_U)\le R-r + \abs{s(\Omega)-s(U)}\le (n+1) d_\mathcal{H}(\Omega, U),
    \end{equation}
   
    where $s(\Omega)$ and $s(U)$ are the Steiner point of $\Omega$ and $U$ respectively.
     Now, using the reverse triangular inequality, the hypothesis \eqref{cond} and \eqref{distpalle}, we have
     \begin{equation*}
        \begin{split}
            d_\mathcal{H}(U,B_U)&\geq d_\mathcal{H}(\Omega,B_U)-d_\mathcal{H}(\Omega,U)\\
            &=d_\mathcal{H}(\Omega,B_\Omega)+d_\mathcal{H}(\Omega,B_U)-d_\mathcal{H}(\Omega,B_\Omega)-d_\mathcal{H}(\Omega,U)\\
            &\geq d_\mathcal{H}(\Omega,B_\Omega)-d_\mathcal{H}(B_U,B_\Omega)-d_\mathcal{H}(\Omega,U)\\
            &\geq d_\mathcal{H}(\Omega,B_\Omega)-(n+2)d_\mathcal{H}(\Omega,U)>\frac{1}{2}d_\mathcal{H}(\Omega,B_\Omega).
        \end{split}
     \end{equation*}
\end{proof}

\subsection{Symmetrization}
\label{symmetrization}
    Let $\Omega$ be an open, bounded and convex set in $\mathbb{R}^n$ and 
    \begin{equation*}
        \mathcal{A}(\Omega)=\{u\in C^\infty(\Omega)\cap C(\overline{\Omega})\,:\, \text{$u$ is strictly convex in $\Omega$ and vanishes on $\partial\Omega$}\}.
    \end{equation*}
    We observe that any $u$ in $\mathcal{A}(\Omega)$ has a unique minimum in $\Omega$. Denote
    \begin{equation*}
        \Omega(\mu)=\{x\in \Omega\,:\, u(x)<\mu\},\quad\partial\Omega(\mu)=\{x\in\Omega\,:\,u(x)=\mu\},
    \end{equation*}
    for $\mu\in[m,0]$, where $m$ is the minimum of $u$. We define the \textit{$k$-symmetrized} of $u$, $k=0,\dots,n-1$, to be
    \begin{equation*}
        u_k^*(x)=\sup\{\mu\in[m,0]\,:\, \zeta_k(\overline{\Omega(\mu)})<|x|\}
    \end{equation*}
    that is a radially symmetric function defined in the ball
    \begin{equation*}
        \Omega_k^*=\{x\,:\, |x|<\zeta_k(\overline{\Omega})\}.
    \end{equation*}

    It is possible to prove that the function $u^\ast_k$ is strictly increasing along the radii and it is strictly convex, so $u_k^* \in \mathcal{A}(\Omega^*_k)$. 

    Moreover, for $k=0$, we recover the standard Schwarz symmetrization.  
    \begin{oss}[Comparison of $u^*_k$] 
    \label{comparison}
        We observe that the Aleksandrov-Fenchel inequality \eqref{Aleksandrov_Fenchel_inequalities}  implies the inclusions 
        \begin{equation*}
            \Omega^*_0\subseteq\Omega^*_1\subseteq\dots\subseteq\Omega^*_k\subseteq\Omega^*_{k+1}\subseteq\dots\subseteq\Omega^*_{n-1}
        \end{equation*}
        and then $u^*_k$ is well defined on the smallest ball $\Omega^*_0$ for all $k=0,\dots,n-1$. If $x\in\Omega^*_0$, then we have
        \begin{equation*}
            \{\mu\,:\,\zeta_{n-1}(\overline{\Omega(\mu)})<|x|\}\subseteq\{\mu\,:\,\zeta_{n-2}(\overline{\Omega(\mu)})<|x|\}\subseteq\dots\subseteq\{\mu\,:\,\zeta_0(\overline{\Omega(\mu)})<|x|\},
        \end{equation*}
        that is, passing to the supremum,
        \begin{equation*}
            0\ge u^*_0(x)\geq u^*_1(x)\geq\dots\geq u^*_k(x)\geq u^*_{k+1}(x)\geq\dots\geq u^*_{n-1}(x)
        \end{equation*}
        for all $x\in\Omega^*_0$. In particular, if $0\leq i<j\leq n-1$, 
        \begin{equation*}
            u^*_i(x)\geq u^*_j(x)
        \end{equation*}
        for all $x\in\Omega^*_i$.
    \end{oss}

\begin{oss} \label{comparison:cones}
    The convexity of the function $u$ and the monotonicity of the quermassintegral under inclusion can be very useful in bounding from below the mean radii of the sublevel sets $\Omega(\mu)$. Indeed, the convexity of $u$ implies that it is possible to find a function $c$ whose graph is the cone of base $\Omega$ and height $m$ that satisfies $u\le c$. If we denote by $C(\mu)=\{c<\mu\}$, it is clear that $C(\mu)\subseteq \Omega(\mu)$.

    Without loss of generality, we can assume that $0\in \Omega$ and $u(0)=m$. By the definition of cone, the sublevel sets of the function $c$ are homothetic to $\Omega$ (see Figure \ref{Figure 1}), and we can explicitly write
   \begin{equation}
    \label{omotetia}
    C(\mu)=\left(1-\frac{\mu}{m}\right)\Omega,
    \end{equation}
 as we are assuming that $u(0)=m$, the center of the homothety is $0$.
 This allows us to have the following bound for the mean radii of the sublevel sets of the function $u$
    \begin{equation}
        \label{control}
        \zeta_k(\Omega(\mu))\geq\zeta_k(C(\mu))=\left(1-\frac{\mu}{m}\right)\zeta_k(\Omega)
    \end{equation}
    where $\mu\in[m,0]$.
\end{oss}

\begin{oss}
The inclusion  $C(\mu)\subset \Omega(\mu)$ is also useful to bound the Hausdorff distance between $\Omega(\mu)$ and $\Omega$
\begin{equation}
\label{1}
    d_\mathcal{H}(\Omega,\Omega(\mu))\leq d_\mathcal{H}(\Omega,C(\mu)),
\end{equation}
and, as $C(\mu)$ is homothetic to $\Omega$, one can bound this distance in terms of $\mu$ and the diameter of $\Omega$, as it holds

$$d_\mathcal{H}(\Omega,t\Omega)\le (1-t)D(\Omega), \quad \forall 0\le t\le 1$$
and so
\begin{equation}
    \label{2}
    d_\mathcal{H}(\Omega,C(\mu))\le \frac{\mu}{m}D(\Omega).
\end{equation}
This will be very useful in order to apply the Propagation Lemma \ref{lembrasco}.
\end{oss}

    \tikzset{every picture/.style={line width=0.75pt}} 
    \begin{figure}
    \centering
        \begin{tikzpicture}[x=0.75pt,y=0.75pt,yscale=-1,xscale=1]

        \draw  [fill={rgb, 255:red, 155; green, 155; blue, 155 }  ,fill opacity=1 ] (81.5,46.5) .. controls (106.5,23.75) and (187.75,20.5) .. (207.75,46.5) .. controls (227.75,72.5) and (222.25,109.25) .. (197.75,119) .. controls (173.25,128.75) and (118.5,116.25) .. (81.5,106.5) .. controls (44.5,96.75) and (56.5,69.25) .. (81.5,46.5) -- cycle ;
        \draw    (219.58,88.42) .. controls (218.84,103.47) and (216.32,123.58) .. (216.21,125.37) .. controls (216.11,127.16) and (205.25,199.5) .. (186.75,230) .. controls (168.25,260.5) and (146.83,261.26) .. (139.83,260.76) .. controls (132.83,260.26) and (103.75,252) .. (80,204.25) .. controls (56.25,156.5) and (57.5,107.75) .. (58.25,80.5) ;
        \draw  [fill={rgb, 255:red, 155; green, 155; blue, 155 }  ,fill opacity=1 ][dash pattern={on 4.5pt off 4.5pt}] (86.15,147.35) .. controls (106.22,137.78) and (121.58,138) .. (142.67,138.5) .. controls (163.75,139) and (196.81,139.25) .. (205.75,149) .. controls (214.69,158.75) and (201.84,199.32) .. (180.25,206.5) .. controls (158.66,213.68) and (119.62,200.62) .. (89.25,193) .. controls (58.88,185.38) and (66.09,156.93) .. (86.15,147.35) -- cycle ;
        \draw    (57.75,89) -- (140.25,259.5) ;
        \draw    (215.89,101.26) -- (139.83,260.76) ;
        \draw  [fill={rgb, 255:red, 191; green, 186; blue, 186 }  ,fill opacity=1 ][dash pattern={on 0.84pt off 2.51pt}] (113.22,157.25) .. controls (124.6,144.76) and (161.59,142.97) .. (170.7,157.25) .. controls (179.81,171.53) and (177.3,191.71) .. (166.15,197.06) .. controls (154.99,202.41) and (130.07,195.55) .. (113.22,190.2) .. controls (96.37,184.84) and (101.84,169.74) .. (113.22,157.25) -- cycle ;
        \draw  [dash pattern={on 0.84pt off 2.51pt}]  (102.25,79.5) -- (102,178.75) ;
        \draw  [dash pattern={on 0.84pt off 2.51pt}]  (176.25,81.5) -- (176,180.25) ;
        \draw  [fill={rgb, 255:red, 191; green, 186; blue, 186 }  ,fill opacity=1 ][dash pattern={on 0.84pt off 2.51pt}] (113.22,57.25) .. controls (124.6,44.76) and (161.59,42.97) .. (170.7,57.25) .. controls (179.81,71.53) and (177.3,91.71) .. (166.15,97.06) .. controls (154.99,102.41) and (130.07,95.55) .. (113.22,90.2) .. controls (96.37,84.84) and (101.84,69.74) .. (113.22,57.25) -- cycle ;

        \draw (75.5,170.4) node [anchor=north west][inner sep=0.75pt]  [font=\tiny]  {$\Omega ( \mu )$};
        \draw (80.5,68.4) node [anchor=north west][inner sep=0.75pt]  [font=\tiny]  {$\Omega $};
        \draw (133.5,263.4) node [anchor=north west][inner sep=0.75pt]  [font=\scriptsize]  {$m$};
        \draw (131,68.65) node [anchor=north west][inner sep=0.75pt]  [font=\tiny]  {$C( \mu )$};

    \end{tikzpicture}
    \caption{Control of the quermassintegrals of the  sublevelsets of $u$ via the quermassintegrals of $\Omega$.}
    \label{Figure 1}
    \end{figure}
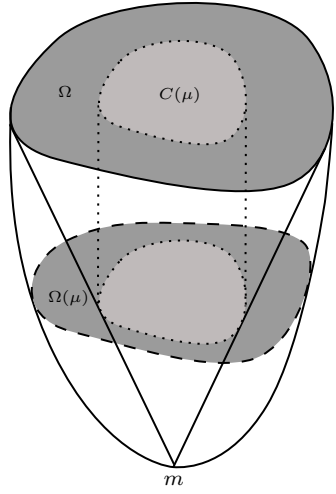
    \subsection{The \texorpdfstring{$k$}{k}-Hessian operator}
Let $\Omega$ be an open subset of $\R^n$ and let $u\in C^2(\Omega)$. As seen in \eqref{kess}, the $k$-Hessian operator is the $k$-th elementary function of the Hessian matrix $D^2u$. Except for the case $k=1$ where the $k$-Hessian operator reduces to the usual Laplace operator, these operators are fully nonlinear and non-elliptic, unless one restricts to the class  of $k$-convex function

$$\mathcal{A}_k(\Omega)= \left\{u\in C^2(\Omega) : S_i (D^2u)\ge 0 \, \text{ in } \, \Omega, \,i=1,\dots, k\right\}.$$

The operator $S_k^{\frac{1}{k}}$ is homogeneous of degree 1, and if we denote by

$$S_k^{ij}(D^2u)=\frac{\partial}{\partial u_{ij}}S_k(D^2u),$$
the Euler identity for homogeneous functions gives
\begin{equation*}
    S_k(D^2u)=\frac{1}{k} S_k^{ij}(D^2u)u_{ij}.
\end{equation*}
A direct computation shows that the $\left(S_k^{1j}(D^2u), \dots, S_k^{nj}(D^2u)\right)$ is divergence-free, hence $S_k(D^2u)$ can be written in divergence form

\begin{equation}
    \label{skdiv}
    S_k(D^2u)=\frac{1}{k}\left(S_k^{ij}(D^2u) u_j\right)_i,
\end{equation}
where the subscripts $i, j$ stand for partial differentiation. 

If $u\in C^2(\Omega)$ and $t$ is a regular point of $u$, on the boundary of $\{u\le t\}$ it is possible to link the $k$-Hessian operator with the $(k-1)$-th curvature

\begin{equation}
    \label{hksk}
   \binom{n-1}{k-1} \sigma_{k-1}=\frac{S_k^{ij}(D^2u)u_iu_j}{\abs{\nabla u}^{k+1}}.
\end{equation}
Finally, we report a classic result without proof, for more details see \cite{tso}, Proposition 6.
\begin{prop}
    \label{prop6}
    Let $\Sigma(t)$ be a smooth family of embedded hypersurfaces in $\mathbb{R}^n$ oriented with respect to the unit inner normal $-\nu$. Then for all $k=0,1,\dots,n-1$ we have
    \begin{equation*}
        \frac{d}{dt}\int_{\Sigma(t)}\sigma_k\, d\mathcal{H}^{n-1}=(n-k-1)\int_{\Sigma(t)}\sigma_{k+1}\,\nu\cdot\xi \, d\mathcal{H}^{n-1}
    \end{equation*}
    where 
    \begin{equation*}
        \xi=\frac{\partial\Sigma(t)}{\partial t}
    \end{equation*}
    is the variation vector field.
\end{prop}
\section{Proofs of main results.}\label{sec3}
\subsection{Proof of Theorem \ref{theorem1.1}}
\label{section3.1}
First of all, we prove the P\'olya-Szeg\H o type quantitative inequality contained in  Theorem \ref{theorem1.1}. Here, the convexity of the setting is crucial, as we want to take advantage of the comparison with cones explained in Remark \ref{comparison:cones} and the Alexandrov-Fenchel inequalities for convex sets. 
\begin{proof}[proof of Theorem \ref{theorem1.1}]
    The chain of inequalities (\ref{1})-(\ref{2}) 
    $$ d_\mathcal{H}(\Omega(\mu),\Omega)\le d_\mathcal{H}(C(\mu),\Omega)\le\frac{\mu}{m}D(\Omega)$$
    gives the following inclusions
    \begin{equation}
    \label{threshold inclusions}
        \begin{split}
            \left\{\mu\in[m,0]\,:\, \mu\geq\frac{m}{2(n+2)}\frac{d_\mathcal{H}(\Omega, B_\Omega)}{D(\Omega)}\right\}&\subseteq\left\{\mu\in[m,0]\,:\, d_\mathcal{H}(\Omega,C(\mu))<\frac{1}{2(n+2)}d_\mathcal{H}(\Omega,B_\Omega)\right\}\\
            &\subseteq\left\{\mu\in[m,0]\,:\, d_\mathcal{H}(\Omega,\Omega(\mu))<\frac{1}{2(n+2)}d_\mathcal{H}(\Omega,B_\Omega)\right\},
        \end{split}
    \end{equation}
    where the smallest set is well defined from (\ref{3}).
    We define
    \begin{equation}
        \label{threshold}
        \overline{\mu}=\inf\left\{\mu\in[m,0]\,:\, \mu\geq\frac{m}{2(n+2)}\frac{d_\mathcal{H}(\Omega, B_\Omega)}{D(\Omega)}\right\}=\frac{m}{2(n+2)}\frac{d_\mathcal{H}(\Omega, B_\Omega)}{D(\Omega)},
    \end{equation}

    and we observe that for all $\mu\in [\overline{\mu}, 0]$, the sublevel set $\Omega(\mu)$ satisfies the hypothesis of  Lemma \ref{lembrasco}, hence
\begin{equation*}
    \alpha_\mathcal{H}(\Omega(\mu))\ge \frac{\alpha_\mathcal{H}(\Omega)}{2}.
\end{equation*}
    
    Proceeding as in \cite{tso}, we have
    \begin{equation*}
        \begin{aligned}
            (k+1)H_k(u,\Omega)&=\int_\Omega (-u)S_k(D^2u)\, dx= \frac{1}{k}\int_\Omega S_k^{ij}(D^2u)u_iu_j\,dx\\
            &=\frac{1}{k}\int_m^0 \int_{\partial\Omega(\mu)} \frac{S_k^{ij}(D^2u)u_iu_j}{\abs{\nabla u}}\,d\mathcal{H}^{n-1}\,d\mu,
        \end{aligned}
    \end{equation*}
    
    where we have applied the divergence theorem and the Coarea formula. We can now apply identity \eqref{hksk} on $\partial\Omega(\mu)$ and the Jensen inequality, obtaining 
     \begin{equation*}
        \begin{aligned}
            k(k+1)H_k(u,\Omega)&=\binom{n-1}{k-1}\int_m^0 \int_{\partial\Omega(\mu)}\sigma_{k-1}\abs{\nabla u}^k\,d\mathcal{H}^{n-1}\,d\mu \ge \binom{n-1}{k-1}\bigintsss_m^0  \dfrac{\left[\int_{\partial\Omega(\mu)}\sigma_{k-1}\,d\mathcal{H}^{n-1}\right]^{k+1}}{\left[\int_{\partial\Omega(\mu)}\frac{\sigma_{k-1}}{\abs{\nabla u}}\,d\mathcal{H}^{n-1}\right]^k}\,d\mu.
        \end{aligned}
    \end{equation*}
    
    As a consequence of proposition \ref{prop6}, it follows

    $$\int_{\partial\Omega(\mu)}\frac{\sigma_{k-1}}{\abs{\nabla u}}\,d\mathcal{H}^{n-1}=n\omega_n \zeta_{k-1}^{n-k}(\Omega(\mu))\frac{d\zeta_{k-1}(\Omega(\mu))}{d\mu},$$
    
   so, recalling the definition of mean radii \eqref{meanradii} and applying the quantitative isoperimetric inequality (\ref{quantitative_mean_radii}) with $i=k-1$ and $j=k$, we have
    \begin{equation*}
        \begin{split}
            &k(k+1)H_k(u;\Omega)\geq \binom{n-1}{k-1}n\omega_n\bigintsss_{m}^{0}\frac{1}{\left(\frac{d\zeta_{k-1}}{d\mu}\right)^k}\left[\frac{\zeta_k^{k+1}(\Omega(\mu))}{\zeta_{k-1}^k(\Omega(\mu))} \right]^{n-k}\,d\mu\\
            &\geq \binom{n-1}{k-1}n\omega_n\bigintsss_{m}^{0}\frac{\zeta_{k-1}^{n-k}(\Omega(\mu))}{\left(\frac{d\zeta_{k-1}}{d\mu}\right)^k}\left[1+\frac{(n+1)\beta_n}{n(n-1)\omega_n}\alpha_\mathcal{H}^\frac{n+3}{2}(\Omega(\mu))\right]^\frac{k+1}{n-k+1}\,d\mu\\
            &\geq \binom{n-1}{k-1}n\omega_n\bigintsss_{m}^{0}\frac{\zeta_{k-1}^{n-k}(\Omega(\mu))}{\left(\frac{d\zeta_{k-1}}{d\mu}\right)^k}\,d\mu\\
            &+\binom{n-1}{k-1}n\omega_n\bigintsss_{m}^{0}\frac{\zeta_{k-1}^{n-k}(\Omega(\mu))}{\left(\frac{d\zeta_{k-1}}{d\mu}\right)^k}\left[\frac{(k+1)(n+1)\beta_n}{2n(n-1)(n-k+1)\omega_n}\alpha_\mathcal{H}^\frac{n+3}{2}(\Omega(\mu))\right]\,d\mu.
        \end{split}
    \end{equation*}
where the last inequality  follows from the inequality
\begin{equation}
\label{5}
    (1+x)^\alpha\ge 1+\frac{\alpha}{2}x, \,\forall\, \alpha>0, \,\forall\, x\in [0,1].
\end{equation}

    Now let us study the two integrals 
separately. The first integral, as also observed in \cite{tso}, is precisely
    \begin{equation*}
        \binom{n-1}{k-1}n\omega_n\bigintsss_{m}^{0}\frac{\zeta_{k-1}^{n-k}(\Omega(\mu))}{\left(\frac{d\zeta_{k-1}}{d\mu}\right)^k}\,d\mu=(k+1)kH_k(u_{k-1}^*;\Omega_{k-1}^*).
    \end{equation*}
    
    We aim to bound from below the second integral in terms of the Hausdorff asymmetry of $\Omega$. 
   Let us consider the threshold $\overline{\mu}$ defined in \eqref{threshold}, for $\mu\in[\overline{\mu}, 0]$ we can apply the propagation Lemma \ref{lembrasco} and the bound (\ref{control}), 
    obtaining
    \begin{equation*}
        \begin{split}
            &\binom{n-1}{k-1}n\omega_n\bigintsss_{m}^{0}\frac{\zeta_{k-1}^{n-k}(\Omega(\mu))}{\left(\frac{d\zeta_{k-1}}{d\mu}\right)^k}\left[\frac{(k+1)(n+1)\beta_n}{n(n-1)(n-k+1)\omega_n}\alpha_\mathcal{H}^\frac{n+3}{2}(\Omega(\mu))\right]\,d\mu\\
            &\geq \kappa_1(n,k)\alpha_\mathcal{H}^\frac{n+3}{2}(\Omega)\bigintsss_{\overline{\mu}}^{0}\frac{\zeta_{k-1}^{n-k}(\Omega(\mu))}{\left(\frac{d\zeta_{k-1}}{d\mu}\right)^k}\,d\mu\geq \kappa_1(n,k)\alpha_\mathcal{H}^\frac{n+3}{2}(\Omega)\bigintsss_{\overline{\mu}}^{0}\left(1-\frac{\mu}{m}\right)^{n-k}\frac{\zeta_{k-1}^{n-k}(\Omega)}{\left(\frac{d\zeta_{k-1}}{d\mu}\right)^k}\,d\mu\\
            &\geq \kappa_1(n,k)\left(1-\frac{\overline{\mu}}{m}\right)^{n-k}\alpha_\mathcal{H}^\frac{n+3}{2}(\Omega)\zeta_{k-1}^{n-k}(\Omega)\bigintsss_{\overline{\mu}}^{0}\frac{1}{\left(\frac{d\zeta_{k-1}}{d\mu}\right)^k}\,d\mu,
        \end{split}
    \end{equation*}
    where the constant is
    \begin{equation*}
        \kappa_1(n,k)=\binom{n-1}{k-1}\left(\frac{1}{2}\right)^\frac{n+3}{2}\frac{(k+1)(n+1)\beta_n}{(n-1)(n-k+1)}.
    \end{equation*}
    To conclude, we perform the change of variable $r=\zeta_{k-1}(\Omega(\mu))$ and a Jensen inequality, and we have
    \begin{equation*}
        \begin{split}
            \kappa_1(n,k)\left(1-\frac{\overline{\mu}}{m}\right)^{n-k}&\alpha_\mathcal{H}^\frac{n+3}{2}(\Omega)\zeta_{k-1}^{n-k}(\Omega)\bigintsss_{\overline{\mu}}^{0}\frac{1}{\left(\frac{d\zeta_{k-1}}{d\mu}\right)^k}\,d\mu\\
            &=\kappa_1(n,k)\left(1-\frac{\overline{\mu}}{m}\right)^{n-k}\alpha_\mathcal{H}^\frac{n+3}{2}(\Omega)\zeta_{k-1}^{n-k}(\Omega)\int_{\zeta_{k-1}(\Omega(\overline{\mu}))}^{\zeta_{k-1}(\Omega)}\left(\left(u^\ast_{k-1}\right)'(r)\right)^{k+1}\, dr\\
            &\geq \kappa_1(n,k)\left(1-\frac{\overline{\mu}}{m}\right)^{n-k}\frac{(-\overline{\mu})^{k+1}}{\zeta_{k-1}^k(\Omega)}\alpha_\mathcal{H}^\frac{n+3}{2}(\Omega)\zeta_{k-1}^{n-k}(\Omega).
        \end{split}
    \end{equation*}
    By \eqref{threshold}, we can bound
    \begin{equation}
    \label{bound 1-mu/m}
        \left(1-\frac{\overline{\mu}}{m}\right)\ge \left(\frac{2n+3}{2(n+2)}\right),
    \end{equation}
    and we obtain
    \begin{equation*}
        H_k(u;\Omega)-H_k(u_{k-1}^*;\Omega_{k-1}^*)\geq \kappa_2(n,k)(-m)^{k+1}\alpha_\mathcal{H}^{\frac{n+3}{2}+k+1}(\Omega)\zeta_{k-1}^{n-2k}(\Omega) \left(\frac{\zeta_{n-1}(\Omega)}{D(\Omega)}\right)^{k+1},
    \end{equation*}
    where the new constant is
    \begin{equation*}
        \kappa_2(n,k)=\frac{\kappa_1(n,k)}{k(k+1)}\left(\frac{2n+3}{2(n+2)}\right)^{n-k}.
    \end{equation*}
    Finally from \eqref{zeta/diam}, the thesis follows.
\end{proof}  

\begin{oss}
    In \cite{Nunziask}, it was observed that the P\'olya-Szeg\H o type inequality
    
    $$H_k(\Omega,u )\ge H_k(\Omega_{k-1}^*, u^*_{k-1})$$
   holds true for every function $u$ with convex sublevel sets. 

   Our result Theorem \ref{theorem1.1} remains valid when the convexity hypothesis is relaxed to the $p-$convexity, for some $p\in (0,1]$, in the sense that we can consider any negative function $u$ such that the power
    \begin{align*}
        -(-u)^p         
    \end{align*}
   is convex.
    A $p$-convex function defined on a convex set has convex sublevel sets, so the P\'olya-Szeg\H o type inequality holds for it, and in this case, Theorem \ref{theorem1.1} reads as follows: let $u\in C^2(\Omega)$ be a negative $p$-convex function, with $p\in (0,1]$, defined on an open, bounded and convex set $\Omega$ which vanishes on $\partial\Omega$, then  there exists a positive constant $\Tilde{C}=\Tilde{C}(n,k,p)$ such that
\begin{equation}
    H_k(\Omega,u)-H_k(\Omega^*_{k-1}, u^*_{k-1} )\ge \Tilde{C}\alpha_\mathcal{H}^{\frac{n+3}{2}+\frac{k+1}{p}}(\Omega) \zeta_{k-1}^{n-2k}(\Omega)\norma{u}_\infty^{k+1},
\end{equation}

and 
\begin{equation}
    \label{Tilde(c)}\Tilde{C}(n,k,p)=\binom{n-1}{k-1}\left(\frac{\omega_{n-1}}{n\omega_n}\right)^\frac{k+1}{p}\frac{(n+1)\beta_n}{2^\frac{n+3}{2}k(n-1)(n-k+1)}\left(\frac{2n+3}{2(n+2)}     \right)^{n-k}.
\end{equation}
\end{oss}

\vspace{2mm}
    Theorem \ref{theorem1.1} leads to a quantitative inequality for the eigenvalue of the $k$-Hessian operator. The eigenvalue of the $k$-Hessian operator is the only real number $\lambda_k(\Omega)$ such that 
    \begin{equation}
   \label{autovalsk}
        \begin{cases}
            S_k(D^2u)=\lambda_k(\Omega)(-u)^{k}& \text{in } \Omega,\\
            u=0& \text{on } \partial\Omega,
        \end{cases}
    \end{equation}
    has a solution (see \cite{Wang}).
    We recall the following variational characterization of the eigenvalue
    \begin{equation*}
        \lambda_k(\Omega)=\inf\left\{\frac{(k+1)H_k(u;\Omega)}{\|u\|^{k+1}_{L^{k+1}(\Omega)}}\,:\, \text{$u\in C^2(\Omega)$, $u$ $k$-convex, $u|_{\partial\Omega}=0$} \right\}.
    \end{equation*}
    As a consequence of the result in \cite{tso} and Remark \ref{comparison} we have that the eigenvalue of the $k$-Hessian operator is minimum on the ball in the class of convex sets for which the eigenfunction has convex level sets, provided that the $(k-1)$-th quermassintegral is fixed. In particular, the following corollary holds if one assumes that there exists a $p$-convex eigenfunction. We recall that evidence of that are contained in \cite{Ma2008TheCO, salani_logcon}.
    \begin{corollario}
    \label{quant:eigen}
        Let $\Omega$ be an open, bounded and convex set of $\mathbb{R}^n$.  Assume that the solution $u$ to \eqref{autovalsk} is $p$-convex, for some $p\in (0,1]$, then 
        there exists a positive constant $C_4=C_4(n,k,p,\Omega)$ such that
        \begin{equation*}
            \lambda_k(\Omega)-\lambda_k(\Omega^*_{k-1})\geq C_4 \alpha^{\frac{n+3}{2}+\frac{k+1}{p}}_\mathcal{H}(\Omega),
        \end{equation*}
        where 
        \begin{equation*}
            C_4(n,k,p,\Omega)=\Tilde{C}(n,k,p)(k+1)\omega_n^{-1} {\zeta_{k-1}^{-2k}(\Omega)}
        \end{equation*}
        and $\Tilde{C}(n,k,p)$ is given in \eqref{Tilde(c)}.
    \end{corollario}
    \begin{proof}
        Let $u$ be the eigenfunction in \eqref{autovalsk}.
        We remark that Cavalieri's principle and Alexandrov-Fenchel inequality \eqref{Aleksandrov_Fenchel_inequalities} imply
        \begin{equation*}
            \int_\Omega (-u)^{k+1}\,dx=(k+1)\int_{0}^{+\infty} t^k|\{-u>t\}|\,dt\leq (k+1)\int_{0}^{+\infty} t^k|\{-u^*_{k-1}>t\}| \,dt=\int_{\Omega^*_{k-1}}(-u^*_{k-1})^{k+1}\,dx.
        \end{equation*}
        Then, from the Remark \ref{comparison} and Theorem \ref{theorem1.1} it follows that
        \begin{equation*}
            \begin{split}
                \lambda_k(\Omega)-\lambda_k(\Omega^*_{k-1})&\geq\frac{(k+1)H_k(u;\Omega)}{\|u\|^{k+1}_{L^{k+1}(\Omega)}}-\frac{(k+1)H_k(u^*_{k-1};\Omega^*_{k-1})}{\|u^*_{k-1}{\|^{k+1}_{L^{k+1}(\Omega^*_{k-1})}}}\\[2ex]
                &\geq (k+1)\frac{H_k(u;\Omega)-H_k(u^*_{k-1};\Omega^*_{k-1})}{\|u^*_{k-1}\|^{k+1}_{L^{k+1}(\Omega^*_{k-1})}}\\[2ex]
                &\geq\frac{(k+1)\Tilde{C}(n,k,p)\zeta_{k-1}^{n-2k}(\Omega)\norma{u}^{k+1}_{L^\infty(\Omega)}}{\|u^*_{k-1}\|^{k+1}_{L^{k+1}(\Omega^*_{k-1})}}\alpha^{\frac{n+3}{2}+\frac{k+1}{p}}_\mathcal{H}(\Omega).  
            \end{split}
        \end{equation*}
Let us observe that the constant can be replaced with a constant independent of $u$, indeed
\begin{equation*}
\begin{aligned}
    \frac{(k+1)\Tilde{C}(n,k,p)\zeta_{k-1}^{n-2k}(\Omega)\norma{u}^{k+1}_{L^\infty(\Omega)}}{\|u^*_{k-1}\|^{k+1}_{L^{k+1}(\Omega^*_{k-1})}}
    \ge& (k+1)\Tilde{C}(n,k,p) \frac{\zeta_{k-1}^{n-2k}(\Omega)}{|\Omega_{k-1}^*|}\\[1.5ex]
    =& (k+1)\Tilde{C}(n,k,p)\omega_n^{-1}{\zeta_{k-1}^{-2k}(\Omega)},
    \end{aligned}
\end{equation*}
   where we have used the inequality
   $$\|u^*_{k-1}\|^{k+1}_{L^{k+1}(\Omega^*_{k-1})}\le \|u^*_{k-1}\|^{k+1}_{L^{\infty}(\Omega^*_{k-1})} |\Omega_{k-1}^\ast|,$$
   and the identity $|\Omega_{k-1}^\ast|=\omega_n \zeta_{k-1}^n(\Omega)$. 
    \end{proof}
    \vspace{3mm}
    Let us define the $k$-Torsional rigidity as
    \begin{equation}
        \label{TorsDef}T(\Omega)=\sup\left\{\frac{\|u\|_{L^1(\Omega)}^{k+1}}{(k+1)H_k(u,\Omega)}\,:\, \text{$u\in C^2(\Omega)$, $u$ $k$-convex, $u|_{\partial\Omega}=0$ }  \right\},
    \end{equation}
    the supremum in \eqref{TorsDef} is achieved by the solution to
    \begin{equation}
    \label{torsDirich}
        \begin{cases}
            S_k(D^2u)= \binom{n}{k}& \text{ in $\Omega$, }\\
            u=0 & \text{ on $\partial \Omega$}.
        \end{cases} 
    \end{equation}
    
    Another consequence of the result in \cite{tso} and Remark \ref{comparison} is that the $k$-Torsional rigidity is maximum on the ball in the class of convex sets for which the solution to \eqref{torsDirich} has convex level sets, provided that the $(k-1)$-th quermassintegral is fixed. Indeed,

    \begin{equation*}
        T(\Omega)= \frac{\left(\int_\Omega -u\,dx \right)^{k+1}}{(k+1)H_k(u;\Omega)}\leq \frac{\left(\int_{\Omega^*_{k-1}} -u^*_{k-1}\,dx \right)^{k+1}}{(k+1)H_k(u^*_{k-1};\Omega^*_{k-1})}\leq T(\Omega^*_{k-1}).
    \end{equation*}
    In particular, the following corollary holds as a consequence of Theorem \ref{theorem1.1} and Remark \ref{comparison}. 
    \begin{corollario}\label{quant:tors}
         Let $\Omega$ be an open, bounded convex set of $\mathbb{R}^n$. If there exists a $p$-convex solution to \eqref{torsDirich}, for some $p\in (0,1]$, then there exists a positive constant $C_5=C_5(n,k,p,\Omega)$ such that
         \begin{equation*}
             \frac{1}{T(\Omega)}-\frac{1}{T(\Omega^*_{k-1})}\geq C_5\alpha^{\frac{n+3}{2}+\frac{k+1}{p}}_\mathcal{H}(\Omega),
         \end{equation*}
        where 
        \begin{equation*}
            C_5(n,k,\Omega)=\frac{k+1}{\omega_n^{k+1}}\Tilde{C}(n,k,p)\zeta_{k-1}^{-k(n+2)}(\Omega),
        \end{equation*}
        and $\Tilde{C}(n,k,p)$ is the constant in \eqref{Tilde(c)}.
    \end{corollario}
    \begin{proof}
        Proceeding as in Corollary \ref{quant:eigen} we have
        \begin{equation*}
            \begin{split}
                \frac{1}{T(\Omega)}-\frac{1}{T(\Omega^*_{k-1})}&\geq\frac{(k+1)H_k(u;\Omega)}{\|u\|^{k+1}_{L^1(\Omega)}}-\frac{(k+1)H_k(u^*_{k-1};\Omega^*_{k-1})}{\|u^*_{k-1}\|^{k+1}_{L^1(\Omega^*_{k-1})}}\\[1.5ex]
                &\geq (k+1)\frac{H_k(u;\Omega)-H_k(u^*_{k-1};\Omega^*_{k-1})}{\|u^*_{k-1}\|^{k+1}_{L^1(\Omega^*_{k-1})}}\\[1.5ex]
                &\geq\frac{(k+1)\Tilde{C}(n,k,p)\zeta_{k-1}^{n-2k}(\Omega)\norma{u}^{k+1}_{L^\infty(\Omega)}}{\|u^*_{k-1}\|^{k+1}_{L^1(\Omega^*_{k-1})}}\alpha^{\frac{n+3}{2}+\frac{k+1}{p}}_\mathcal{H}(\Omega).
            \end{split}
        \end{equation*}
        Finally, since
        \begin{equation*}
            \|u^*_{k-1}\|_{L^1(\Omega^*_{k-1})}^{k+1}\leq \|u^*_{k-1}\|_{L^\infty(\Omega^*_{k-1})}^{k+1}|\Omega^*_{k-1}|^{k+1}=\|u\|^{k+1}_{L^\infty(\Omega)}\omega_n^{k+1}\zeta_{k-1}^{n(k+1)}(\Omega),
        \end{equation*}
        we have
        \begin{equation*}
            \frac{1}{T(\Omega)}-\frac{1}{T(\Omega^*_{k-1})}\geq \frac{k+1}{\omega_n^{k+1}}\Tilde{C}(n,k,p)\zeta_{k-1}^{-k(n+2)}(\Omega)\alpha_\mathcal{H}^{\frac{n+3}{2}+\frac{k+1}{p}}(\Omega).
        \end{equation*}
    \end{proof}
        
\subsection{Proof of Theorem \ref{teo1.2}}
\label{section 3.2}
We now proceed with the proof of the second result.

\begin{proof}[proof of Theorem \ref{teo1.2}]
    Let $f^0(|x|)$ be the Schwartz rearrangement of the function $f$ given in problem (\ref{poisson:sk}). 
    Let $u$ be the solution to (\ref{poisson:sk}) in $\mathcal{A}(\Omega)\cap C^2(\overline{\Omega})$, $u_{k-1}^*$ be its $(k-1)$-symmetrized and let $u^0$ be the solution to (\ref{poisson:symm}).
    We recall the explicit integral representation of the radial solution $u^0$:
    \begin{equation}
        \label{integral:representation}
        u^0(x)=-\left(\frac{k}{\binom{n-1}{k-1}}\right)^\frac{1}{k}\int_{|x|}^{R}\left(r^{-n+k}\int_{0}^{r}f^0(s)s^{n-1}\,ds\right)^\frac{1}{k}\,dr,
    \end{equation}
    where $R$ is the radius of the ball in which $u^0$ is defined.

     We integrate both sides of the equation \eqref{poisson:sk} on the sublevel set $\Omega(\mu)$, and we proceed as in $\S\ref{section3.1}$, obtaining

     \begin{equation*}
         \begin{aligned}
            k \int_{\Omega(\mu)} f(x)\, dx=& k\int_{\Omega(\mu)} S_k(D^2u)\, dx =\binom{n-1}{k-1}\int_{\partial\Omega(\mu)} \sigma_{k-1}\abs{\nabla u}^{k}\, d\mathcal{H}^{n-1}\ge\\
            \ge & \frac{1}{\left(\frac{d\zeta_{k-1}}{d\mu}\right)^k}\left[\frac{\zeta_k^{k+1}(\Omega(\mu))}{\zeta_{k-1}^k(\Omega(\mu))} \right]^{n-k} \ge \binom{n-1}{k-1}n\omega_n\frac{\zeta_{k-1}^{n-k}(\Omega(\mu))}{\left(\frac{d\zeta_{k-1}}{d\mu}\right)^k}\left(1+\kappa_3(n,k)\alpha_\mathcal{H}^\frac{n+3}{2}(\Omega(\mu))\right),
         \end{aligned}
     \end{equation*}
     where in the last inequality we have applied
     the quantitative quermassintegral inequality  \eqref{quantitative_mean_radii} and the constant is
     
    \begin{equation*}
        \kappa_3(n,k)=\frac{(n+1)(k+1)\beta_n}{2n(n-1)(n-k+1)\omega_n}.
    \end{equation*}

     On the other hand, we can apply the Hardy-Littlewood inequality and  we have
    \begin{equation*}
        \begin{split}
            &k\int_{\Omega(\mu)}f\,dx
            \leq k\int_{0}^{|\Omega(\mu)|}f^0\left(\left(\frac{a}{\omega_n}\right)^\frac{1}{n}\right)\,da\leq\\ &k\int_{0}^{\omega_n\zeta_{k-1}^n(\Omega(\mu))}f^0\left(\left(\frac{a}{\omega_n}\right)^\frac{1}{n}\right)\,da=kn\omega_n\int_{0}^{\zeta_{k-1}(\mu)}f^0(r)r^{n-1}\,dr.
        \end{split}
    \end{equation*}
    If we rewrite the previous inequality in a more compact form, we have
    \begin{equation*}
        1+\kappa_3(n,k)\alpha_\mathcal{H}^\frac{n+3}{2}(\Omega(\mu))\leq \frac{k}{\binom{n-1}{k-1}}\left(\frac{d\zeta_{k-1}}{d\mu}\right)^k\zeta_{k-1}^{-n+k}(\Omega(\mu))\int_{0}^{\zeta_{k-1}(\mu)}f^0(r)r^{n-1}\,dr,
    \end{equation*}
    and if we erase both members to the power $1/k$,  we get
    \begin{equation*}
        1+\frac{\kappa_3(n,k)}{2k}\alpha_\mathcal{H}^\frac{n+3}{2}(\Omega(\mu))\leq \frac{d\zeta_{k-1}}{d\mu}\left(\frac{k}{\binom{n-1}{k-1}}\right)^\frac{1}{k}\left(\zeta_{k-1}^{-n+k}(\Omega(\mu))\int_{0}^{\zeta_{k-1}(\mu)}f^0(r)r^{n-1}\,dr \right)^\frac{1}{k},
    \end{equation*}
    that can be rewritten, thanks to (\ref{integral:representation}), as
    \begin{equation}\label{intermedii}
        1+\frac{\kappa_3(n,k)}{2k}\alpha_\mathcal{H}^\frac{n+3}{2}(\Omega(\mu))\leq \frac{d\zeta_{k-1}}{d\mu}\, (u^0)'(\zeta_{k-1}(\Omega(\mu))).
    \end{equation}
    Now, let us integrate \eqref{intermedii} between $\mu$ and $0$, with 
    \begin{equation}
        \label{set:u^*}
        \mu\in\{\nu\in[m,0]\,:\,\zeta_{k-1}(\Omega(\nu))<|x|\},
    \end{equation}
    we have
    \begin{equation*}
        \begin{split}
            \mu&\geq u^0(\zeta_{k-1}(\Omega(\mu)))+\int_{\mu}^{0}\kappa_4(n,k)\alpha_\mathcal{H}^\frac{n+3}{2}(\Omega(\nu))\,d\nu\\
            &\geq  u^0(\zeta_{k-1}(\Omega(\mu)))+\int_{u_{k-1}^*}^{0}\kappa_4(n,k)\alpha_\mathcal{H}^\frac{n+3}{2}(\Omega(\nu))\,d\nu,
        \end{split}
    \end{equation*}
    where the constant is
    \begin{equation*}
        \kappa_4(n,k)=\frac{\kappa_3(n,k)}{2k}.
    \end{equation*}
    If we pass to the supremum on the set (\ref{set:u^*}), by  the definition of $(k-1)$-symmetrized of $u$, we have
    \begin{equation}
    \label{parz:parz}
        \begin{split}
            u_{k-1}^*(x)&\geq u^0(\zeta_{k-1}(u_{k-1}^*(x)))+\int_{u_{k-1}^*}^{0}\kappa_4(n,k)\alpha_\mathcal{H}^\frac{n+3}{2}(\Omega(\nu))\,d\nu\\
            &\geq u^0(x)+\int_{u_{k-1}^*}^{0}\kappa_4(n,k)\alpha_\mathcal{H}^\frac{n+3}{2}(\Omega(\nu))\,d\nu,
        \end{split}
    \end{equation}
    here we used the inequality $u^0(\zeta_{k-1}(u_{k-1}^*(x)))\geq u^0(x)$ which holds bacause $u^0$ is increasing. 
   
    Once again, we want to bound the integral in the right-hand side of \eqref{parz:parz} in terms of the Hausdorff asymmetry of $\Omega$. 
As in the proof of Theorem \eqref{theorem1.1}, we consider the threshold (\ref{threshold}), so we can apply the propagation Lemma \ref{lembrasco}, obtaining
    \begin{equation}\label{quant:parz}
        \begin{split}
            u_{k-1}^*(x)-u^0(x)&\geq \int_{u_{k-1}^*}^{0}\kappa_4(n,k)\alpha_\mathcal{H}^\frac{n+3}{2}(\Omega(\nu))\,d\nu\\
            &\geq \int_{\max\{u_{k-1}^*(x),\overline{\mu}\}}^{0}\kappa_4(n,k)\alpha_\mathcal{H}^\frac{n+3}{2}(\Omega(\nu))\,d\nu\\
            &\geq \kappa_5(n,k)\alpha_\mathcal{H}^\frac{n+3}{2}(\Omega)\left(-\max\{\overline{\mu}, u_{k-1}^*(x)\}\right),
        \end{split}
    \end{equation}
    where 
    \begin{equation*}
        \kappa_5(n,k)=\kappa_4(n,k)2^{-\frac{n+3}{2}}.
    \end{equation*}
    From the last inequality, we find that
    \begin{equation}\label{parz}
        \|u_{k-1}^*-u^0\|_{L^\infty\left(\Omega_{k-1}^*\right)}\geq \kappa_5(n,k)\alpha_\mathcal{H}^\frac{n+3}{2}(\Omega)(-\overline{\mu}),
    \end{equation}
    that  can be rewritten thanks to (\ref{threshold}) as
    \begin{equation*}
        \|u_{k-1}^*-u^0\|_{L^\infty\left(\Omega_{k-1}^*\right)}\geq \kappa_5(n,k)\alpha_\mathcal{H}^\frac{n+5}{2}(\Omega)\left(-\frac{m}{2(n+2)}\frac{\zeta_{n-1}(\Omega)}{D(\Omega)} \right).
    \end{equation*}
    Finally, from \eqref{zeta/diam} the thesis follows.
\end{proof}
\begin{oss}
    In \cite{ABMP}[Theorem 1.2], the authors proved the analogous of our  Theorem \ref{teo1.2} in the case of the Laplacian, obtaining

    \begin{equation}\label{goal0}
    \dfrac{\norma{u^{*}_{0}-u^0}_\infty}{\abs{\Omega}^{\frac{2-n}{n}} \norma{f}_1}\geq \widetilde C \alpha^3_\mathcal{F}(\Omega),
        \end{equation}
    where $\alpha^3_\mathcal{F}(\Omega)$ is the \emph{Fraenkel asymmetry} of $\Omega$.
    In that case, once a threshold is defined and one obtains a partial result as in \eqref{parz}, it is more subtle to conclude and obtain \eqref{goal0}. In our case, the comparison with cones in Remark \ref{comparison:cones} allows us to conclude more easily.
\end{oss}

\subsection{Proof of Theorems \ref{gradienti} and \ref{teo1.3}}
Now, we want to take advantage of inequality \eqref{poitso} to prove the comparison on the $k$-Hessian integral $H_k$ contained in Theorem \ref{gradienti}.

\begin{proof}[Proof of Theorem \ref{gradienti}]
    The proof is very similar to the proof of the Hardy-Littlewood inequality, indeed it follows from Cavalieri's principle and Fubini's Theorem

    \begin{equation*}
        \begin{aligned}
            \int_{\Omega}(-u) S_k(D^2u)\, dx&=\int_{\Omega}(-u) f\, dx= \int_{\Omega}\left( \int_0^{f(x)}\, ds\right) \left( \int_0^{-u(x)}\, dt\right)\, dx\\
            &=\int_0^{+\infty} \int_0^{+\infty} \left(\int_\Omega \chi_{\{f>s\}}(x)\chi_{\{-u>t\}}(x)\, dx\right)\, ds\, dt\\
            &= \int_0^{+\infty} \int_0^{+\infty} \abs{\{f>s\}\cap \{-u>t\} }\, ds\, dt \\
            &\le \int_0^{+\infty} \int_0^{+\infty} \min\{\abs{\{f>s\}}, \abs{\{-u>t\} }\}\, ds\, dt.
        \end{aligned}
    \end{equation*}
    Let us observe that the Alexandrov-Fenchel inequalities \eqref{Aleksandrov_Fenchel_inequalities} imply
    $$\abs{-u>t}\le \omega_n \left(\frac{W_{k-1}(-u>t)}{\omega_n}\right)^{\frac{n}{n-k+1}}=\omega_n \left(\frac{W_{k-1}(-u^*_{k-1}>t)}{\omega_n}\right)^{\frac{n}{n-k+1}}=\abs{-u^*_{k-1}>t},$$
    where the last equality follows from the fact that $\{-u^*_{k-1}>t\}$ is a ball, while from the definition of the Schwarz symmetrization, we have

    $$\abs{\{f>s\}}=\abs{\{f^0>s\}},$$
    moreover, the sets $\{-u^*_{k-1}>t\}$ and $\{f^0>s\}$ are concentric balls, and
    this implies

     \begin{equation*}
        \begin{aligned}
            \int_{\Omega}(-u) S_k(D^2u)\, dx &\le \int_0^{+\infty} \int_0^{+\infty} \min\{\abs{\{f>s\}}, \abs{\{-u>t\} }\}\, ds\, dt\\
            &\le \int_0^{+\infty} \int_0^{+\infty} \min\{|\{f^0>s\}|, \abs{\{-u^*_{k-1}>t\} }\}\, ds\, dt\\
            &=\int_{\Omega_{k-1}^*} (-u^*_{k-1})f^0\, dx \le\int_{\Omega_{k-1}^*} (-u^0)f^0\, dx\\
            &=\int_{\Omega_{k-1}^*} (-u^0)S_k(D^2u^0)\, dx.
        \end{aligned}
    \end{equation*}
\end{proof}

We are now ready for the proof of Theorem \ref{teo1.3}. Here, we make use of what we proved in Theorems \ref{teo1.2} and \ref{gradienti}.

\begin{proof}[proof of Theorem \ref{teo1.3}]
    From the proof of Theorem \ref{gradienti} it follows that

    $$\int_{\Omega}(-u) S_k(D^2u)\, dx \le \int_{\Omega_{k-1}^*} (-u^*_{k-1})f^0\, dx,$$
    that combined with \eqref{quant:parz} gives

    \begin{equation*}
    \begin{aligned}
        H_k(u^0,\Omega_{k-1}^*)- H_k(u,\Omega)&\ge H_k(u^0,\Omega_{k-1}^*)-\int_{\Omega_{k-1}^*} (-u^*_{k-1})f^0\, dx=\int_{\Omega_{k-1}^*} (u^*_{k-1}-u^0)f^0\, dx\\&\ge \kappa_5(n,k) \alpha_{\mathcal{H}}^\frac{n+3}{2}(\Omega)\int_{\Omega^*_{k-1}}f^0 (-\max\{\overline{\mu}, u_{k-1}^*\})\, dx.
        \end{aligned}
    \end{equation*}
    As the final step, we need to bound the right-hand side, and this can be done by considering that we are integrating a positive function

\begin{equation*}
    \begin{aligned}
        H_k(u^0,\Omega_{k-1}^*)- H_k(u,\Omega)&\ge  \kappa_5(n,k) \alpha_{\mathcal{H}}^\frac{n+3}{2}(\Omega)\int_{u^*_{k-1}<\overline{\mu}}f^0 (-\overline{\mu})\, dx\\
        &= \kappa_5(n,k) \alpha_{\mathcal{H}}^\frac{n+3}{2}(\Omega) (-\overline{\mu})\int_{0}^{\abs{u^*_{k-1}<\overline{\mu}}}  f^0\left(\left(\frac{a}{\omega_n}\right)^{\frac{1}{n}}\right)\, da. 
        \end{aligned}
    \end{equation*}
    where the last equality holds as both $u_{k-1}^*$ and $f^0$ are radially symmetric functions.
Let us recall that the function
\[
\varphi:s\rightarrow \dfrac{1}{s}\int_0^sf^0 
 \]
 is decreasing as $f^0$ is decreasing, so we can bound
$$\int_{0}^{\abs{u^*_{k-1}<\overline{\mu}}}  f^0\left(\frac{a}{\omega_n}\right)^{\frac{1}{n}}\, da \ge \frac{\abs{u^*_{k-1}<\overline{\mu}}}{\omega_n\zeta_{k-1}^n(\Omega)} \int_{0}^{\omega_n \zeta_{k-1}^n(\Omega)}  f^0\left(\frac{a}{\omega_n}\right)^{\frac{1}{n}}\, da. $$
 
   Moreover, the measure  of the sublevel set can be bounded, thanks to \eqref{control} and\eqref{bound 1-mu/m}
    \begin{equation*}
        \begin{split}
            |u^*_{k-1}-\overline{\mu}|&=W_0(\Omega^*_{k-1}(\overline{\mu}))=\omega_n\zeta_0^n(\Omega^*_{k-1}(\overline{\mu}))=\omega_n\zeta_{k-1}^n(\Omega^*_{k-1}(\overline{\mu}))\geq \omega_n \left(1-\frac{\overline{\mu}}{m} \right)^n\zeta_{k-1}^n(\Omega)\\
            &\geq \omega_n \left(\frac{2n+3}{2(n+2)} \right)^n\zeta_{k-1}^n(\Omega),
        \end{split}
    \end{equation*}
    and then
    \begin{equation*}
        H_k(u^0,\Omega_{k-1}^*)- H_k(u,\Omega)\geq \kappa_5(n,k) \alpha_{\mathcal{H}}^\frac{n+3}{2}(\Omega) (-\overline{\mu})  \left(\frac{2n+3}{2(n+2)} \right)^n\int_{0}^{\omega_n \zeta_{k-1}^n(\Omega)}  f^0\left(\left(\frac{a}{\omega_n}\right)^{\frac{1}{n}}\right)\,da
    \end{equation*}
    As in the proof of Theorems \ref{theorem1.1} and \ref{teo1.2}, we can bound the threshold $\overline{\mu}$ thanks to \eqref{threshold} and \eqref{zeta/diam}, so to obtain
    \begin{equation*}
        -\overline{\mu}\geq \frac{\|u\|_{L^\infty(\Omega)}}{2n(n+2)}\frac{\omega_{n-1}}{\omega_n}\alpha_\mathcal{H}(\Omega),
    \end{equation*}
     and observing that
     $$\int_{0}^{\omega_n \zeta_{k-1}^n(\Omega)}  f^0\left(\left(\frac{a}{\omega_n}\right)^{\frac{1}{n}}\right)\,da\ge \norma{f}_{L^1(\Omega)},$$
     then the proof is complete.
\end{proof}
\begin{oss}
Let us observe that also Theorem \ref{teo1.2} and Theorem \ref{teo1.3} hold more in general if we assume that the solution to  (\ref{poisson:sk}) is $p$-convex, for some $p\in (0,1]$.

Theorem \ref{teo1.2} reads as follows: there exists a positive constant $\Tilde{C}_2=\Tilde{C}_2(n,k,p)$ such that
    \begin{equation*}
        \frac{\|u_{k-1}^*-u^0\|_{L^\infty(\Omega_{k-1}^*)}}{\|u\|_{L^\infty(\Omega)}}\geq \Tilde{C}_2\alpha_\mathcal{H}^\frac{p(n+3)+2}{2p}(\Omega),
    \end{equation*}
    where
    \begin{equation*}
        \Tilde{C}_2(n,k)=\frac{(n+1)(k+1)\beta_n}{2^\frac{n+7}{2}nk(n-1)(n-k+1)\omega_n}\left(\frac{\omega_{n-1}}{2(n+2)n\omega_n)}\right)^\frac{1}{p}.
    \end{equation*}

On the other hand, Theorem \ref{teo1.3} can be generalized as follows: there exists a positive constant $\Tilde{C}_3=\Tilde{C}_3(n,k,p)$ such that
    \begin{equation*}
        \frac{H_k(u^0; \Omega^*_{k-1})-H_k(u;\Omega)}{\norma{u}_{L^\infty(\Omega) }\norma{f}_{L^1(\Omega)}}\ge \Tilde{C}_3\alpha_\mathcal{H}^{\frac{n+3}{2}+\frac{1}{p}}(\Omega),
    \end{equation*}
    where
    \begin{equation*}
        \Tilde{C}_3(n,k)=\frac{(n+1)(k+1)\beta_n}{2^\frac{n+7}{2}nk(n-1)(n-k+1)\omega_n}\left(\frac{\omega_{n-1}}{2(n+2)n\omega_n)}\right)^\frac{1}{p}.
    \end{equation*}
\end{oss}

 \subsection*{Acknowledgements}
The authors were partially supported by Gruppo Nazionale per l’Analisi Matematica, la Probabilità e le loro Applicazioni
(GNAMPA) of Istituto Nazionale di Alta Matematica (INdAM).   
Alba Lia Masiello was partially supported by  PRIN 2022, 20229M52AS: "Partial differential equations and related geometric-functional
inequalities," CUP:E53D23005540006.

\bibliographystyle{plain}
\bibliography{biblio}

\Addresses

\end{document}